\documentclass[11pt]{amsart}
\usepackage{amsmath, amssymb, amscd, url, pstricks, pst-node, pst-coil,
  graphicx, multido, verbatim}
\usepackage[run=2]{auto-pst-pdf}
\usepackage{amsfonts, amsthm, hyperref}
\usepackage{color}
\usepackage{enumerate}

\usepackage{hyperref}
\hypersetup{colorlinks,%
citecolor=red,%
linkcolor=blue,%
}
\usepackage{color}

\newrgbcolor{lightblue}{0.5 0.5 1.0}
\newrgbcolor{lightgreen}{0.5 1.0 0.5}
\newrgbcolor{deepred}{0.5 0.1 0.1}
\newrgbcolor{darkgreen}{0.1 0.4 0.1}
\newrgbcolor{deepblue}{0.1 0.1 0.5}

\DeclareMathOperator{\vol}{vol}

\theoremstyle{plain}
\newtheorem{theorem}[equation]{Theorem}
\newtheorem{lemma}[equation]{Lemma}

\newtheorem{proposition}[equation]{Proposition}
\newtheorem{corollary}[equation]{Corollary}

\theoremstyle{definition}
\newtheorem{example}[equation]{Example}
\newtheorem{question}[equation]{Question}
\newtheorem{problem}[equation]{Problem}

\newtheorem{definition}[equation]{Definition}

\theoremstyle{remark}
\newtheorem{remark}[equation]{Remark}

\numberwithin{equation}{section}
\newtheorem*{acknowledgements}{Acknowledgements}

\def\Chi{\raisebox{2pt}{$\chi$}}
\def\Z{\mathbb Z}

\def\C{\mathbb C}
\def\R{\mathbb{R}}

\def\P{\mathcal{P}}
\def\z{\mathbf{z}}

\def\T{\mathbb T}
\def\WG{\widehat{G}}
\def\WGR{\widehat{G}}
\def\WGL{\widehat{G}_\mathcal{L}}
\def\WG#1{\widehat{G}_{\mathcal{L}_{#1}}}
\def\e{\varepsilon}

\def\v{\mathbf{v}}
\def\w{\mathbf{w}}
\def\u{\mathbf{u}}

\def\y{\mathbf{y}}
\def\t{\mathbf{t}}

\def\Rv{R^{\v}}

\newcommand{\LL}{\mathcal{L}}
\newcommand{\TT}{\mathcal{T}}

\title{Tiling the integer lattice with translated sublattices}

\makeatletter
\@namedef{subjclassname@2010}{\textup{2010} Mathematics Subject Classification}
\makeatother

\author{Maciej Borodzik}
\thanks{The first author was supported by  Polish OPUS grant No 2012/05/B/ST1/03195}
\address{Institute of Mathematics, University of Warsaw, ul. Banacha 2,
02-097 Warsaw, Poland}
\email{mcboro@mimuw.edu.pl}

\author{Danny Nguyen}
\address{Department of Mathematics,  
University of California, Los Angeles, 
520 Portola Plaza,
Los Angeles, CA 90095}
\email{ldnguyen@math.ucla.edu}

\author{Sinai Robins}

\thanks{The second and third authors were supported in part by the Singapore
  MOE Tier 2 research grant MOE2011-T2-1-090}

\address{Mathematics Department,  
Brown University,
Box 1917, 
151 Thayer Street,
Providence, RI 02912} \email{sinai.robins@gmail.com}

\begin{document}
\begin{abstract}
  When $\Z^d$ is represented as a finite disjoint union of translated integer
  sublattices, the translated sublattices must possess some special
  properties. Such a representation is called a \emph{lattice tiling}. We develop a
  theoretical framework, based on multiple residues and dual groups, to provide
  a set of necessary and sufficient conditions for such a lattice tiling to exist.  We
  also investigate the question of when a lattice tiling must possess at least
  two translated sublattices which are translates of one another.
\end{abstract} 
\maketitle

\section{Introduction}
\subsection{Overview}
A \emph{sublattice} is a full-rank subgroup of $\Z^d$. Suppose we decompose
$\Z^d$ into a finite, disjoint union of integer translates of sublattices. We call such a 
decomposition of the integer lattice $\Z^d$ a \emph{lattice tiling}.  Given a lattice tiling, what can be said about the structure of the translated sublattices?  

A lattice tiling has the \emph{translational property} if at least two of its
translated sublattices are different cosets of
the same sublattice. To motivate the results of this paper, we focus on the following three questions:

\begin{question}\label{question1}
What are some natural and general \emph{necessary} and \emph{sufficient} conditions for the existence of a lattice tiling?  
\end{question}

\begin{question}\label{question2}
In any general dimension $d$, are there some nice sufficient conditions for a lattice tiling to have the translational property?
\end{question}

\begin{question}\label{question3}
For  $d=2$, is there a lattice tiling which does not have the translational property?
\end{question}

In the process of trying to answer these questions, we develop some analytic tools which may be of independent interest. These tools involve generating functions associated to sublattices of $\Z^d$, residue calculus of holomorphic functions of several variables  \cite{Tsikh}, and some elementary considerations.  

\subsection{Terminologies and notations}
Before stating the main results we layout some concise definitions of notions that will be used throughout the paper.

\smallskip
The index of a sublattice $\LL$ in $\Z^d$ is called the \emph{determinant}
of $\LL$, denoted by $\det \LL$. 
Throughout the article,  we write 
$d-$dimensional vectors in bold, to distinguish them from scalars. 
Thus any vector $\v \in \Z^d$ has coordinates $(v_1,\dots,v_d)$. Furthermore, we write $\v\ge 0$
if $v_1,\dots,v_d\ge 0$. Define $\mathbf{0}:= (0,\ldots,0)$ and
$\mathbf{1}:=(1,\ldots,1)$.

\smallskip
For any sublattice $\LL \subseteq \Z^d$ and integer vector $\v \in \Z^d$, we call the discrete set of vectors 
\[\v+\LL := \{ \v + \w \mid \w \in \LL\}\] 
a \emph{translate} of $\LL$.  The vector $\v$ will be referred to as a \emph{translate vector}.  Thus, a more formal description of a lattice tiling is the existence of a collection of  ($d-$dimensional) sublattices $\LL_1,\dots,\LL_n$ and translate vectors 
$\v_1,\dots,\v_n \in \Z^d$ such that 
\[
\bigcup_{j=1}^n \{  \v_j+\LL_j \}=\Z^d,
\]
and such that $\{ \v_i+  \LL_i  \}  \cap \{ \v_j+\LL_j \}= \emptyset$ for all $i \not= j$.   
In other words, for any $\w\in\Z^d$ there exists a unique $j\in\{1,\dots,n\}$ such that $\w-\v_j\in\LL_j$.

\smallskip
Let $\T^d=\{(z_1,\dots,z_d)\in\C^d\colon |z_1|=\dots=|z_d|=1\}$.
For any $\z=(z_1,\dots, z_d) \in \T^d$ and $\v = (v_1,\dots,v_d) \in \Z^d$, we define
\[
\Chi_\z (\v) :=z_1^{v_1}\dots z_d^{v_d} \in \T.
\]
Such a homomorphism $\Chi_\z : \Z^d \to \T$
is refered to as a \emph{character}.  
We also say that a character $\Chi_\z$ has \emph{finite order} if each $z_j$ is
additionally assumed to be a root of unity, a condition tantamount to saying
that the image of $\Chi_{\z}$ is finite. 
\smallskip

For a sublattice $\LL\subseteq\Z^d$, we call a  complex point $\z\in\T^d$ a \emph{dual point} of $\LL$ if $\Chi_\z(\v)=1$ for all $\v \in \LL$. Since a point $\z\in \T^d$ is a dual point of $\LL$  if and only if 
the homomorphism $\Chi_{\z}$ restricted to $\LL$ is trivial, the dual points can be regarded as characters on the finite abelian group 
\[
G_\LL:=\Z^d/\LL,
\]
which we call the \emph{group of the sublattice}. If $G_\LL$ is cyclic then $\LL$
is called a \emph{cyclic sublattice}.
It is a standard fact that the dual points form a group, known as the Pontryagin dual to $G_\LL$, 
and it is particularly useful  that this group is isomorphic to $G_\LL$. We
denote the dual group by $\WGL$. We clearly have
\[|\WGL|=|G_\LL|=\det \LL.\]

\subsection{Statement of results}

Many necessary conditions can be deduced about the structure of a lattice tiling by elementary
arguments. One notable condition among such various results that we will prove in Subsection~\ref{simpleconditions} is: 

\begin{theorem}\label{thm:prime}
In a lattice tiling, if there is a prime $p$ such that $p^k$ divides the determinant of one of the sublattice translates, then $p^k$ divides the determinant of another sublattice translate. 
\end{theorem}

\noindent
To give an answer to Question~\ref{question2}, we have the following result, stated in terms of cyclic sublattices.

\begin{theorem}\label{cycliccondition}
If we have a lattice tiling in which the sublattice translate with largest determinant is cyclic, then our lattice tiling has the translational property. 
\end{theorem}

This will be proved after Corollary~\ref{cor:comeinpairs}. In dimension 2 we
can prove a stronger condition, made precise in
Theorem~\ref{thm:main-two}. Finally, the following necessary and sufficient
conditions answer question \eqref{question1}, which we call the character
formulas. This result came with the help of generating functions.

\begin{theorem}\label{mainequality}
We have a lattice tiling with $\v_1+\LL_1,\dots,\v_n+\LL_n$ if and only if for any character of finite order
$\Chi_\z$, the following holds
\begin{equation}\label{eq:mainresintro}
\sum_{    j \colon \z\in\WG{j}    } 
\frac{\Chi_\z (\v_j)}{\det \LL_j}=\begin{cases} 1,&\textrm{if } \z=(1,\ldots,1)\\ 0, & \textrm{otherwise.}\end{cases}
\end{equation}
\end{theorem}
The `only if' part of Theorem~\ref{mainequality} is stated and proved in  Proposition~\ref{prop:mainequality}. The `if' part is 
Theorem~\ref{thm:sufficient}. The proof will also illustrate that
~\eqref{eq:mainresintro} is indeed a finite set of conditions.


\smallskip
The layout of the following sections are as follows. Section~\ref{sec:two}
displays some basic properties of lattice tilings and also the definition of
our generating functions. Section~\ref{sec:three} examines dual points and
characters in connection with these generating
functions. Section~\ref{sec:rescalc} brings in the main analytic ingredients
to prove Theorem~\ref{mainequality}. Although Question~\ref{question3}
remains open, Section~\ref{sec:translational} will illustrate the restrictive
nature of 2-dimensional lattice tilings. Lastly, we will pose two
open questions in Section~\ref{sec:open}.

\subsection{Background review}\label{sec:background}
Note that equation \eqref{eq:mainresintro} is a basic relation between roots
of unity, with rational coefficients.  Indeed, $1$-dimensional lattice
tilings have been extensively studied using vanishing sums of roots of unity
over the rationals.  In dimension $1$, a lattice tiling is also known in the
literature as a Disjoint Covering System (DCS).  Paul Erd{\H{o}}s initiated
the study of covering systems in general (which means that the arithmetic
progressions may not necessarily be disjoint, see \cite{Erd}), and Erd{\H{o}}s
credits the beautiful proof of  the translational property, in dimension $1$,
to an unpublished  paper  by Mirsky and Newman, and independently to an
unpublished paper of Davenport and  Rado.  Many interesting papers have since
been written about the $1$-dimensional case of lattice tilings, and for more
background, including some fascinating results on vanishing sums of roots of
unity, the reader may refer to   \cite{DCS2},  \cite{DCS3}, \cite{BS1},
\cite{Redei}, \cite{Schinzel}, \cite{Sun1}, \cite{Znam}.  

There is a related question in the context of a general group $G$.  Suppose that $G$ is  partitioned into some cosets of some  of its subgroups. 
Then the conjecture, known as the Herzog-Sch\"{o}nheim conjecture, says that
at least two of the cosets must have the same index. The Herzog-Sch\"{o}nheim conjecture was solved for finite nilpotent groups in
1986, in the paper \cite{Herzog-SchonheimConj}. Since
then stronger results have been proved (see \cite{Sun2}), but the general case
still remains open.    

Since a lattice tiling may also be thought of as a covering of the group $\Z^d$ by a finite disjoint union of cosets of subgroups of 
$\Z^d$, it then follows from  \cite{Herzog-SchonheimConj}, in our abelian
group setting, that any lattice tiling must contain at least two sublattice
translates of the same determinant.  In other words, there are two sublattice
translates that must have the same volume for their fundamental domain.  But
almost any other question about these fundamental domains remains open.

In the recent paper  \cite{Tiling1}, the translational property for higher
dimensional lattice tilings was considered from a discrete Fourier
perspective.  The translational property for dimension $d>1$ was apparently
first considered in another unpublished manuscript, this time an MIT Master's
thesis by A. Schwartz \cite{Schwartz}, using purely combinatorial methods. We
note that in general, the translational property does not hold for dimension
$d > 2$ (see Example~\ref{ex:nonstandard}).

Finally, we mention that Paul Erd{\H{o}}s \cite{Erd2} himself has been quoted as saying in 1995 that ``Perhaps my favorite problem of all concerns covering systems''.

\begin{acknowledgements}
We thank Professor Krzysztof Przeslawski for bringing the Master's thesis \cite{Schwartz} to our attention. 
The project was started when M.B. and S.R. were visiting Renyi Institute in Budapest. We would like to thank the Renyi Institute for their hospitality.  The authors S.R. and D.N. would like to thank the University of Warsaw for their hospitality, where this project was continued.
\end{acknowledgements}

\newpage
\section{Sublattices, generating functions, and characters}\label{sec:two}

\subsection{The tiling condition}\label{simpleconditions}
\begin{definition}
A tiling $\v_1+\LL_1,\dots,\v_n+\LL_n$ of $\Z^d$ \emph{splits} if $\{1,\dots,n\}$ can be partitioned into $I_1\cup\ldots\cup I_k$ with $k>1$, $|I_j|>0$ for
$j=1,\ldots,k$, $|I_1|>1$, so that for any $j=1,\ldots,k$ the union
\[
\bigcup_{i\in I_j} \{        \v_i+\LL_i   \}
\]
is another sublattice translate. Otherwise a tiling is called \emph{primitive}. This definition simply captures the intuition that one can generate more complex lattice tilings by starting with a simple one, and splitting one of the existing sublattice translates into new ``coarser'' sublattice translates.
\end{definition}

\begin{lemma}\label{lem:prime}
Assume that $\LL\subset \Z^d$ is a sublattice with $\det \LL=p$, for $p$ a prime.  Fix any integer vector  $\v\notin\LL$. 
Then we have $\mathrm{span}\{\LL,\v\} = \Z^d$.
\end{lemma}

\begin{proof}
Let $\TT$ be the sublattice generated by $\LL$ and $\v$. Since $\v\notin\LL$,
we have $\LL\subsetneq \TT \subseteq \Z^d$ and the index of $\TT$ in $\Z^d$ is therefore a proper divisor of $p$. We conclude that $\det\TT=1$ and hence $\TT=\Z^d$.
\end{proof}

\begin{proposition}\label{prop:split}
If we have $\det \LL_k=p$ with $p$ a prime for some $k$ then the tiling either splits or all the sublattices in the tiling are equal to $\LL_k$.
\end{proposition}
\begin{proof}
Assume that $\det\LL_1=p$ and translate the tiling if necessaray so that
$\v_1=0$. By this assumption, we have  $\LL_1 \cap (\v_i+\LL_i) = \emptyset$
for all $i>1$. So $\v_i\notin\mathrm{span}\{\LL_1,\LL_i\}$ and hence $\mathrm{span}\{\LL_1,\LL_i\} \subsetneq \Z^d$. From Lemma~\ref{lem:prime}, we conclude that $\LL_i\subseteq\LL_1$ for all $i>1$ and this implies $\v_i+\LL_i$ lies in the translate of $\LL_1$ under $\v_i$. 

Since $\Z^d/\LL_1$ is a group with $p$ elements, it is isomorphic to $\Z/p\Z$.
Let us fix an isomorphism $\Z^d/\LL_1\to\Z/p\Z$ and define $\pi\colon\Z^d\to\Z^d/\LL_1\stackrel{\cong}{\to}\Z/p\Z$ as a composition of
the projection map with this isomorphism.  For any $k\in\{0,\ldots,p-1\}$, we define the subset of indices
\[I_k=\{i\in\{1,\ldots,n\}\colon \pi(\v_i)=k\}.\]
The union
\[ 
   \bigcup_{i\in I_k}   \{   \v_i+\LL_i    \}
\]
is equal to $\pi^{-1}(k)$. Therefore, it is equal to a translate of $\LL_1$. 
If $|I_k|>1$ for somke $k$, we have a split tiling. If for all $k$ we have $|I_k|=1$, then all sublattices are equal
to $\LL_1$.
\end{proof}

\begin{lemma}\label{lem:nocop}
Assume that the sublattices $\LL_1,\dots,\LL_n$ have coprime determinants, i.e. 
\[\gcd(\det \LL_1,\dots,\det \LL_n)=1.\] 
Then the sublattice $\mathcal{L}$ generated by  $\bigcup \LL_i$, over the integers, 
is isomorphic to $\Z^d$.
\end{lemma}
\begin{proof}
Let $\LL$ be the sublattice generated by $\LL_1,\dots,\LL_n$. Repeating the argument in Lemma~\ref{lem:prime}, we see that $\det\LL \mid \det \LL_i$ for $1\leq i \leq n$.
Thus we have $\det\LL \mid \gcd(\det \LL_1,\dots,\det \LL_n) = 1$ and consequently $\LL = \Z^d$.
\end{proof}

\begin{proposition}\label{prop:nocoprime}
Let $\v_1+\LL_1,\dots,\v_n+\LL_n$ be a lattice tiling. Then for any $i$ and $j$ we have $\gcd(\det \LL_i,\det \LL_j)>1$.
\end{proposition}
\begin{proof}
Assume that $\gcd(\det \LL_1,\det \LL_2)=1$. By Lemma~\ref{lem:nocop}, $\LL_1
\cup \LL_2$  generates $\Z^d$ and so any integer vector $\v \in \Z^d$ can be written as $\w_1-\w_2$ for $\w_1\in \LL_1$ and $\w_2\in \LL_2$. Representing $\v_2-\v_1$ in this form, we have $\v_1+\w_1=\v_2+\w_2\in(\v_1+\LL_1)\cap(\v_2+\LL_2)\neq\emptyset$. We obtain a contradiction.
\end{proof}

\begin{lemma}\label{lem:volume}
In a lattice tiling, we have the following relation
\begin{equation}\label{eq:density}
\sum_{i=1}^n\frac{1}{\det \LL_i}=1.
\end{equation}
\end{lemma}
\begin{proof}
We observe that the number of integer points of $\LL_i$ in a cube $[-N,N]^d$
is equal to $(2N)^d/\det \LL_i+O(N^{d-1})$.
\end{proof}

This ``density'' results will turn out to be a subcase of \eqref{eq:mainres}
with $\z=\mathbf{1}$. Nevertheless, this already allows us to prove
Theorem~\ref{thm:prime}  from the introduction.

\begin{proof}[Proof of Theorem~\ref{thm:prime}]
Assume $p^k \mid \det{\LL_i}$ and $p^k \nmid \det{\LL_j}$ for $j \neq
i$. Equation ~\eqref{eq:density} is now
\[
1 = \frac{1}{\det{\LL_i}} + \sum_{j\neq i}\frac{1}{\det{\LL_j}} = \frac{1}{p^k
  a} + \frac{c}{p^k b}
\]
where $p \nmid b$ and $p \mid c$. Rewrite this as $p^k ab = b + ac$ and
contradition follows from divisibility by $p$.
\end{proof}

\subsection{Generating functions}

We define one of our main objects of study, namely a generating function that is attached to each sublattice translate of a lattice tiling. 
\begin{definition}\label{def:generating}
Let $\v+\LL$ be a sublattice translate, with $\LL$ any integer sublattice of $\Z^d$ and $\v$ any integer vector.  
We define its \emph{generating function}  by
\[
\Theta_{\LL+\v}(\z) := 
\sum_{\substack{\w\in \LL + \v \\ \w \geq 0}}  \z^{\w} = \sum_{\substack{\w\in \LL + \v \\ \w \geq 0}} z_1^{w_1} \dots z_d^{w_d}.
\]
\end{definition}

Note the important fact that we are restricting our summation to the positive orthant. This makes the series absolutely
convergent for all $\z \in \C^d$ such that $|z_j|<1$ for all $1\leq j \leq d$.  Our next step is
to give an algorithm for computing $\Theta_{\LL+\v}$ and to show that it is in fact a rational function on $\C^d$. 
To this end we introduce another definition.

\begin{definition}
Let $\LL$ be a sublattice. We define  $t_1,\dots,t_d$ as the minimal positive integers such that $(0,\dots,0,t_j,0,\dots)\in \LL$.
These integers are called the \emph{polar values} of $\LL$.
\end{definition}

\begin{example}
Assume that $\LL$ is a sublattice in dimension $2$ spanned by the vectors $(a,b)$ and $(c,d)$. Let us  define  $\tilde{e}=\gcd(a,c)$ 
and $\tilde{f}=\gcd(b,d)$. Then it is routine to see that
\[t_1=\frac{|ad-bc|}{\tilde{f}},\ t_2=\frac{|ad-bc|}{\tilde{e}}.\]
\end{example}

It easy to see that if $(a,b)$ and $(c,d)$ span $\LL$, then the determinant of $\LL$ is $|ad-bc|$. Hence the polar values divide the determinant, a
fact that we can prove in any dimension.

\begin{lemma}\label{lem:polardivide}
The polar values divide the determinant. 
\end{lemma}
\begin{proof}
Let $e_1=(1,0,\dots,0)$. 
The fact that $t_1$ is a polar value means that $t_1e_1\in \LL$ and for any $k=1,\dots,t_1-1$, $ke_1\not\in \LL$. 
Consider the subgroup of $G_\LL=\Z^d/\LL$ spanned by the image of $e_1$.  We see that its order is  $t_1$, so it divides the order of $G_{\LL}$.
\end{proof}

\begin{lemma}\label{lem:numofpoints}
Let $S$ be the half open cube
\begin{equation}\label{eq:S}
S=\{(x_1,\dots,x_d)\colon 0\le x_i<t_i\}.   
\end{equation}
Then we have
\begin{equation}\label{eq:numinS}
\#(S\cap \LL)=\frac{t_1t_2\dots t_d}{\det\LL}.
\end{equation}
\end{lemma}
\begin{proof}
Let $\{ \v_1,\dots,\v_d \}$ be a basis of $\LL\subseteq\Z^d$, and consider the map $J\colon\R^d\to\R^d$ defined by
\begin{equation}\label{eq:changeofvars}
J(x_1,\dots,x_d)=x_1 \v_1 +\dots+ x_d\v_d.
\end{equation}
Then $J$ is a linear map and $J(\Z^d)=\LL$. Let $T=J^{-1}(S)$. Then $T$ is a half-open parallelepiped with integral corners. We have
\[\#(S\cap \LL)=\#(T\cap\Z^d)=\vol(T)=\frac{\vol S}{\det J}=\frac{t_1\dots t_d}{\det \LL}.\]
\end{proof}

\begin{proposition}\label{prop:R-compute}
Let $\v+\LL$ be a sublattice translate and 
$t_1$,\dots, $t_d$ be the polar values of $\LL$.  Then
\[    \Theta_{\LL+\v}(\z)    =\frac{\Rv(\z)}{(1-z_1^{t_1})\dots(1-z_d^{t_d})},\]
where
\begin{equation}\label{eq:Rv}
\Rv(\z)=\sum_{\w\in S\cap(\v+\LL)} \z^\w.
\end{equation}
\end{proposition}
\begin{proof}
Let $\LL_S$ be the sublattice spanned by vectors $(t_1,0,\dots,0),\dots,(0,\dots,t_d)$. It is clear that
\[\Theta_{\LL_S}(\z)=\frac{1}{(1-z_1^{t_1})\dots(1-z_d^{t_d})}.\]
Now we have
\[
\v+\LL=\bigcup_{\w\in S\cap(\v+\LL)} \{    \w+\LL_S    \}.
\]
Hence
\[\Theta_{\LL+\v}=\sum_{\w\in S} \z^{\w} \Theta_{\LL_S}(\z)=\frac{\Rv(\z)}{(1-z_1^{t_1})\dots(1-z_d^{t_d})}.\]
\end{proof}

\begin{remark}
If the sublattice translate is in fact a sublattice, i.e. if $\v=\mathbf{0}$,
then we write $R(\z)$ instead of $\Rv(\z)$. Similarly, we write $\Theta_{\LL}$
instead of $\Theta_{\LL+\v}$ if $\v=0$. These notations will be used later in Lemma~\ref{lem:later}.
\end{remark}

\begin{example}
Consider a sublattice $\LL=(2\Z\times 2Z) \cup [(1,1)+(2\Z\times 2\Z)]=\{(x,y)\colon x+y=0\bmod 2\}$. The generating function of the sublattice $2\Z\times 2\Z$
is clearly $\frac{1}{(1-x^2)(1-y^2)}$. We get therefore
\[\Theta_\LL(x,y)=\frac{1+xy}{(1-x^2)(1-y^2)}.\]
\end{example}

\begin{example}\label{ex:figure1}
Assume that $\LL$ is spanned by  $\v_1=(4,1)$ and $\v_2=(2,3)$. Then $t_1=10$ and $t_2=5$.
In this example we see that $S\cap \LL=\{(0,0),(2,3),(4,1),(6,4),(8,2)\}$ (the marked points on Figure~\ref{fig:figure1}). 
Hence $R(x,y)=1+x^2y^3+x^4y^1+x^6y^4+x^8y^2$.
\end{example}

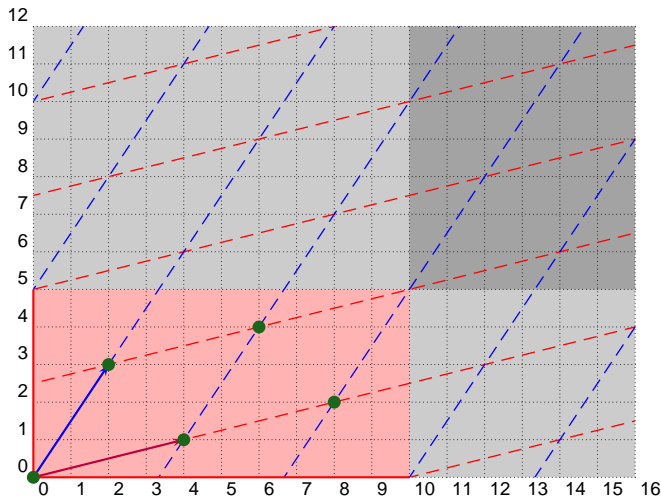
\begin{figure}
\begin{pspicture}(-5,-4)(5,4)
\pspolygon*[linestyle=none,fillcolor=blue,opacity=0.2](1,-3)(4,-3)(4,3)(1,3)
\pspolygon*[linestyle=none,fillcolor=green,opacity=0.2](-4,-0.5)(4,-0.5)(4,3)(-4,3)
\pspolygon*[linestyle=dotted,linecolor=red,fillcolor=red,opacity=0.3](-4,-3)(-4,-0.5)(1,-0.5)(1,-3)
\rput(-4,-3){\psscalebox{0.5}{\psgrid[subgriddiv=1,gridlabels=14pt,griddots=10](0,0)(16,12)}}

\begin{psclip}{\pspolygon[linestyle=none](-4,-3)(4,-3)(4,3)(-4,3)}
\multido{\rsx=-16+1,\rsy=-11+1.5,\rex=12+1,\rey=-4+1.5}{10}{\psline[linecolor=red,linestyle=dashed, linewidth=0.5pt](\rsx,\rsy)(\rex,\rey)}
\multido{\rsx=-16+2,\rsy=-11+0.5,\rex=-5+2,\rey=5.5+0.5}{14}{\psline[linecolor=blue,linestyle=dashed, linewidth=0.5pt](\rsx,\rsy)(\rex,\rey)}
\end{psclip}
\psline[linecolor=red](-4,-3)(-4,-0.5)\psline[linecolor=red](-4,-3)(1,-3)
\psline[linecolor=purple]{->}(-4,-3)(-2,-2.5)
\psline[linecolor=blue]{->}(-4,-3)(-3,-1.5)
\pscircle[fillcolor=darkgreen,fillstyle=solid,linestyle=none](-4,-3){0.1}
\pscircle[fillcolor=darkgreen,fillstyle=solid,linestyle=none](-2,-2.5){0.1}
\pscircle[fillcolor=darkgreen,fillstyle=solid,linestyle=none](0,-2){0.1}
\pscircle[fillcolor=darkgreen,fillstyle=solid,linestyle=none](-3,-1.5){0.1}
\pscircle[fillcolor=darkgreen,fillstyle=solid,linestyle=none](-1,-1){0.1}
\end{pspicture}
\caption{A sublattice spanned by $(4,1)$ and $(2,3)$. Here the polar values are $t_1 =  10$ and $t_2 = 5$.  See Example~\ref{ex:figure1}.}\label{fig:figure1}
\end{figure}

\bigskip
\section{More on dual points and characters}\label{sec:three}
Let us recall that $\z \in \T^d$ is called a dual point of $\LL$ if
$\Chi_\z(\w)=\z^{\w}=1$ for all $\w\in\LL$. Here $\T^d = \{\z \in \C^d : |z_i| =
1\}$.
The following ``orthogonality relation'' for the characters $\Chi_\z$ is
well-known. But as it is crucial to our proofs, we include the proof for the sake of completeness.

\begin{lemma}\label{prop:sigmazero}
Let $\LL\subseteq\Z^d$ be a sublattice with a basis $\v_1,\dots,\v_d$ and the fundamental
parallelepiped $\P=\{ \lambda_1\v_1+\dots+\lambda_d\v_d: 0\leq\lambda_i < 1 \}\subset\R^d$. If $\z=(z_1,\dots,z_d)$ is a dual point different from $\mathbf{1}=(1,\dots,1)$, then
\[\sum_{\w\in\Z^d\cap\P}\Chi_\z (\w)=0.\]
\end{lemma}
\begin{proof}
Observe that the elements of $\P\cap \Z^d$ are in one-to-one correspondence
with elements of the quotient group $\Z^d/\LL$. The character $\Chi_\z$
can now be regarded as a character on the group $\Z^d/\LL$. This character is non-trivial 
because $\z\neq\mathbf{1}$. Now we use the standard fact that the average of a non-trivial character over a compact group (in particular over
a finite group) is zero, and we are done.
\end{proof}

\begin{example}
Consider again $\LL$ generated by $\v_1=(4,1)$ and $\v_2=(2,3)$. 
The points in $\P$ are $(0,0)$, $(1,1)$, $(2,1)$, $(3,1)$, 
$(2,2)$, $(3,2)$, $(4,2)$, $(3,3)$, $(4,3)$  and $(5,3)$;
see Figure~\ref{fig:figure2}.
We consider $\z=(\e,\e)$, where $\e=e^{2\pi i/5}$. Then $\Chi_\z$ has the following values at these points:
$\e^0,\e^2,\e^3,\e^4,\e^4,\e^0,\e^1,\e^1,\e^2$ and $\e^3$. The sum is clearly $0$.
\end{example}

\begin{figure}
\begin{pspicture}(-5,-3)(5,3)
\pspolygon*[linestyle=dotted,linecolor=blue,fillcolor=lightblue,opacity=0.1](-4,-2)(0,-1)(2,2)(-2,1)
\psline[linestyle=solid,linecolor=blue]{->}(-4,-2)(0,-1)
\psline[linestyle=solid,linecolor=blue]{->}(-4,-2)(-2,1)
\rput(-4,-2){\psgrid[subgriddiv=1,gridlabels=7pt,griddots=10](0,0)(8,4)}
\pscircle[fillcolor=red,fillstyle=solid](-4,-2){0.1}
\pscircle[fillcolor=red,fillstyle=solid](-3,-1){0.1}
\pscircle[fillcolor=red,fillstyle=solid](-2,-1){0.1}
\pscircle[fillcolor=red,fillstyle=solid](-1,-1){0.1}
\pscircle[fillcolor=red,fillstyle=solid](-2,0){0.1}
\pscircle[fillcolor=red,fillstyle=solid](-1,0){0.1}
\pscircle[fillcolor=red,fillstyle=solid](0,0){0.1}
\pscircle[fillcolor=red,fillstyle=solid](-1,1){0.1}
\pscircle[fillcolor=red,fillstyle=solid](0,1){0.1}
\pscircle[fillcolor=red,fillstyle=solid](1,1){0.1}
\end{pspicture}
\caption{The points in the parallelepiped $P$ for the sublattice spanned by $(4,1)$ and $(2,3)$.}\label{fig:figure2}
\end{figure}
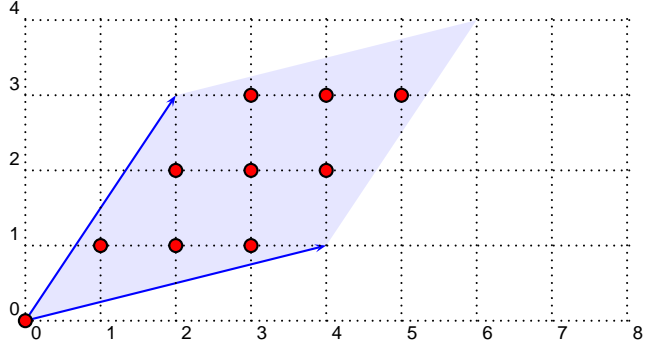

The next two results will also be very useful for us. We first prove it for sublattices, namely when the translate vector is $(0,\dots,0)$.

\begin{lemma}\label{lem:iszero}
Suppose that $\z=(z_1,\dots,z_d)\in\T^d$ satisfies $z_1^{t_1}=\dots=z_d^{t_d}=1$ and $\z$ is \emph{not} a dual point of $\LL$ (the $t_i$ are the polar values). 
Then $R(\z)=0$.
\end{lemma}
\begin{proof}
By definition $\z$ is a dual point of the sublattice $\LL_S$, defined to be the integral span of the 
vectors $(t_1,0,\dots,0),\dots,(0,\dots,0,t_d)$.
The map $J$ defined in \eqref{eq:changeofvars} induces a dual map $J^*\colon\T^d\to\T^d$ by the formula
\[\Chi_{J^*\z}(\w)=\Chi_{\z}(J\w).\]
Let $\y=J^*\z$, $\TT=J^{-1}(\LL_S)$ and $\P(\TT)$ the fundamental parallelepiped of $\TT$.
By definition, $J^*$ maps dual points of $\LL_S$ to dual points of $J^{-1}(\TT)$.
We have
\[
R(z_1,\dots,z_d) = \sum_{\t \in S \cap \LL} \Chi_\z (\t) = \sum_{\w \in
  \P(\TT)} \Chi_\z (J\w) = \sum_{\w \in \P(\TT)}\Chi_\y (\w).
\]

Now $\y$ is a dual point of $\TT$ and $y \neq \mathbf{1}$, for otherwise $\z$ would be a dual point of $\LL$.
The result follows by applying Lemma~\ref{prop:sigmazero}.
\end{proof}

The version for a translated sublattice is:

\begin{lemma}\label{lem:later}
Let $\LL$ be a sublattice with polar values $(t_1,\dots,t_d)$.
Suppose that $\z=(z_1,\dots,z_d)$ satisfies $z_1^{t_1}=\dots=z_d^{t_d}=1$. Then for any translate vector $\v$ we have
\[\Rv(\z)=\Chi_\z(\v)R(\z).\]
\end{lemma}
\begin{proof}
Let $S$ be the cube as defined in \eqref{eq:S}. We have:
\[
R^{\v}(\z) = \sum_{\substack{\w \in \LL \\ \w + \v \in S}} \Chi_{\z} (\w + \v)
= \Chi_{\z}(\v) \sum_{\substack{\w \in \LL \\ \w + \v \in S}} \Chi_{\z}(\w)
= \Chi_{\z}(\v) \sum_{\w \in \LL \cap (S-\v)} \Chi_{\z}(\w). 
\]
Observer that for any $\w \in \LL \cap (S-v)$, there is a unique $\t_\w \in \LL_S$
such that $\w + \t_\w \in S$. The points in $\LL \cap S$ and $\LL \cap (S-v)$
are actually in one-to-one correspondence. Because $\z$ is a dual point of $\LL_s$, we have $\Chi_{\z}(\w + \t_\w) = \Chi_{\z}(\w)
\Chi_{\z}(\t_\w) = \Chi_{\z}(\w)$, and hence
\[
\sum_{\w \in \LL \cap (S-\v)} \Chi_{\z}(\w) = 
\sum_{\w \in \LL \cap (S-\v)} \Chi_{\z}(\w + \t_\w) = 
\sum_{\w \in \LL \cap S} \Chi_{\z}(\w) = R(\z).
\]
\end{proof}

We finish this section with another important result, whose proof follows immediately from Pontryagin duality.

\begin{lemma}
Assume that $\LL\subset\Z^d$ is a sublattice and $\v\in \Z^d\setminus \LL$. Then there exists a dual point $\z\in\WGL$
such that $\Chi_\z (\v)\neq 1$.
\end{lemma}

\begin{corollary}\label{cor:subset}
If we are given two sublattices $\LL_1$ and $\LL_2$ and each dual point of $\LL_1$ is a dual point of $\LL_2$ then $\LL_2\subseteq \LL_1$. 
In particular, if $\WG{1} = \WG{2}$ as subgroups of $\T^d$, then $\LL_1 = \LL_2$.
\end{corollary}
\begin{proof}
Assume that $\v\in \LL_2\setminus \LL_1$. Then there exist a dual point $\z$ of $\LL_1$ such that $\Chi_\z (\v)\neq 1$.
But as $\z$ is a dual point of $\LL_2$ by assumption, we get $\Chi_\z (\v)=1$, a contradiction.
\end{proof}

\bigskip
\section{Residue calculus}\label{sec:rescalc}
\subsection{Residues of generating functions}
Assume now that we have sublattices $\LL_1,\dots,\LL_n$ in $\Z^d$. Denote the polar values of the sublattice $\LL_i$ by
$t_{i,1},\dots,t_{i,d}$. Assume that there exist non-negative vectors $\v_1,\dots,\v_n$ such that we have the lattice tiling
\begin{equation}\label{eq:tiling}
\begin{cases}
\{ \v_1+\LL_1\} \cup \{ \v_2+\LL_2 \} \cup\dots\cup \{ \v_n+\LL_n\}=\Z^d,\\
\{\v_i+\LL_i\}\cap\{\v_j+\LL_j\}=\emptyset & \textrm{for $i\neq j$.}
\end{cases}
\end{equation}
Let $\Theta_i$ be the generating function for $\v_i+\LL_i$. 
\begin{lemma}\label{lem:fund}
We have the following equality of rational functions, with $\z=(z_1,\ldots,z_d)$:
\begin{equation}\label{eq:fund}
\sum_{i=1}^n\Theta_i(\z)=\frac{1}{(1-z_1)\cdots(1-z_d)}.
\end{equation}
Moreover, the identity \eqref{eq:fund} is equivalent to the condition that $\v_1+\LL_1$, \dots, $\v_n+\LL_n$ provide a lattice tiling of $\Z^d$.
\end{lemma}
\begin{proof}
The right-hand side of \eqref{eq:fund} is the generating function for the tiling consisting of a single sublattice $\Z^d$. Let us choose $\z$
such that $|z_j|<1$ for all $j$. By the tiling condition we have
\begin{equation}\label{eq:sumation}
\sum_{\substack{\w\in\Z^d\\\w\geq 0}} \z^{\w} =
\sum_{j=1}^n\sum_{\substack{\w\in\v_j+\LL_j\\\w\geq 0}} \z^{\w},
\end{equation}
where we use the fact that we can change the summation order of absolutely convergent power series. Equation~\eqref{eq:sumation}
is equivalent to \eqref{eq:fund} for $\z$ such that $|z_j|<1$ (see Definition~\ref{def:generating}). Both sides of \eqref{eq:fund}
are rational functions agreeing on an open subset of $\C^d$, so they are equal.
It is also clear that \eqref{eq:sumation} implies the tiling condition.
\end{proof}

\medskip
We now fix some $\z\in\C^d$.
Consider a torus $T=\{\mathbf{u}\in\C^d\colon |z_1-u_1|=\ldots=|z_d-u_d|=\e\}$ for $\e>0$ sufficiently small. Equation \eqref{eq:fund} implies that
\begin{equation}\label{eq:integrals}
\sum_{j=1}^n\int_{T}\Theta_j(u)d\mathbf{u}=\int_T\frac{d\mathbf{u}}{(1-u_1)\dots(1-u_d)}.
\end{equation}
Here $d\u=du_1\wedge\ldots\wedge du_d$ is the volume form on $T$ (note that \eqref{eq:integrals} can be regarded as comparison of multidimensional residues,
see \cite{Tsikh}). We want to study the integrals appearing on the left-hand side of \eqref{eq:integrals}.
To this end recall that by Proposition~\ref{prop:R-compute} that we can write
\[\Theta_i(\z)=\frac{R_i(\z)}{(1-z_1^{t_{i,1}})\cdots(1-z_d^{t_{i,d}})},\]
where
\[
R_i(\z) = \sum_{\w \in S_i \cap (\v_i + \LL_i)} \Chi_{\z}(\w).
\]

\begin{lemma}\label{lem:stupidintegral}
For any $j=1,\ldots,n$, the integral
\[\int_{T}\frac{R_j(\mathbf{u})}{(1-u_1^{t_{j,1}})\dots (1-u_d^{t_{j,d}})}d\mathbf{u}\]
is zero unless $\z \in \widehat{G}_{\LL_j}$. In the latter case it is equal to $ \frac{(-2\pi i)^d}{\det \LL_j} \Chi_\z (\v_j+\mathbf{1})$.
\end{lemma}
\begin{proof}
As $R_j$ is analytic at $\z$, we have
\begin{equation}\label{eq:int}
\int_T\frac{R_j(\u)}{(1-u_1^{t_{j,1}})\dots (1-u_d^{t_{j,d}})}d\mathbf{u}=
R_j(\z)\int_T\frac{d\mathbf{u}}{(1-u_1^{t_{j,1}})\dots(1-u_d^{t_{j,d}})}.
\end{equation}
But
\[\int_T\frac{d\mathbf{u}}{(1-u_1^{t_{j,1}})\dots(1-u_d^{t_{j,d}})}=\prod_{k=1}^d\int_{|u_k-z_k|=\e}\frac{du_k}{1-u_k^{t_{j,k}}}.\]
By Goursat's lemma the integrals on the right-hand side vanish unless $z_1^{t_{j,1}}=\dots=z_d^{t_{j,d}}=1$. So assume that $z_1^{t_{j,1}}=\dots=z_d^{t_{j,d}}=1$.
Since for any $x_0\in\C$ and any integer $m>0$
\[\int_{|x-x_0|=\varepsilon}\frac{dx}{1-x^{m}} \stackrel{x=zx_0}{=} x_0 \int_{|z-1|=\varepsilon}\frac{dz}{1-z^{m}}=-2\pi i\frac{x_0}{m},
\]
we have
\begin{equation}\label{eq:prod}
\prod_{k=1}^d\int_{|u_k-z_k|=\e}\frac{du_k}{1-u_k^{t_{j,k}}}=\frac{(-2\pi i)^d}{t_{j,1}\dots t_{j,d}}z_1\dots z_d.
\end{equation}

Given that $z_1^{t_{j,1}}=\dots=z_d^{t_{j,d}}=1$, to compute $R_j(\z)$ in this case we use Lemma~\ref{lem:iszero} and Lemma~\ref{lem:later}. 
We get $R_j(\z)=0$ unless $\z$ is a dual point of $L_j$. In the latter case, by
Lemma~\ref{lem:later} and \eqref{eq:numinS}, we have
\begin{equation}\label{eq:rj}
R_j(\z)=\frac{t_{j,1}\dots t_{j,d}}{\det \LL_j} \Chi_\z (\v_j).
\end{equation}
Substituting the $R_j$ from \eqref{eq:rj} and the integral from \eqref{eq:prod} into \eqref{eq:int} we conclude the proof.
\end{proof}

We now combine Lemma~\ref{lem:stupidintegral} with Lemma~\ref{lem:fund} to obtain our main technical result.

\begin{proposition}\label{prop:mainequality}
Let $\z \in \T^d$ be arbitrary. In a lattice tiling, we have
\begin{equation}~\label{eq:mainres}
\sum_{j\colon \z\in\WG{j}}\frac{\Chi_\z (\v_j)}{\det \LL_j}=\begin{cases} 1,&\textrm{if }\z=\mathbf{1}\\ 0 & \textrm{otherwise.}\end{cases}
\end{equation}
\end{proposition}
\begin{proof}
The case $\z = \mathbf{1}$ is Lemma~\ref{lem:volume}. If $\z \neq \mathbf{1}$,
consider \eqref{eq:integrals}. The integrals on the left-hand side can be
computed with
Lemma~\ref{lem:stupidintegral}. The integral on the right-hand side is $0$.
The proposition follows immediately.
\end{proof}

From Proposition~\ref{prop:mainequality} we can deduce an immediate corollary.
\begin{corollary}\label{cor:comeinpairs}
Given a lattice tiling, assume that $\z$ is a dual point of $\LL_i$ and $\z\neq\mathbf{1}$. Then there exists at least one other
 $\LL_j$ such that $\z$ is also a dual point of $\LL_j$.
\end{corollary}
\begin{proof}
If $\z$ belongs only to $\WG{i}$, then the left-hand side of \eqref{eq:mainres} is $\frac{\Chi_\z(\v_i)}{\det\LL_i}\neq 0$, and  we obtain a contradiction.
\end{proof}

\begin{proof}[Proof of Theorem~\ref{cycliccondition}]
Assume $\LL_1$ is cyclic and has largest determinant. Let $\z$ be an element in $\WG{1}$ of order $\det \LL_1$.
By Corollary~\ref{cor:comeinpairs} there 
exists $j>1$ such that $\z\in\WG{j}$. 
But then the whole group generated by $\z$ lies in $\WG{j}$, and hence $\WG{1}\subseteq\WG{j}$. By maximality of $\det \LL_1$ we have
$\WG{1}=\WG{j}$. Now Corollary~\ref{cor:subset} implies that $\LL_1=\LL_j$.
\end{proof}

\begin{example}\label{ex:nonstandard}
Consider the four sublattices $\LL_1=(2\Z\times 2\Z\times \Z)$, $\LL_2=(2\Z\times \Z\times 2\Z)$, $\LL_3=(\Z\times 2\Z\times 2\Z)$
and $\LL_4=(2\Z\times 2\Z\times 2\Z)\cup [(1,1,1)+(2\Z\times 2\Z\times 2\Z)]$.
It is known that $(1,0,0)+\LL_1$, $(0,0,1)+\LL_2$, $(0,1,0)+\LL_3$ and $\LL_4$ tile $\Z^3$. We have
\begin{multline*}
\frac{z_1}{(1-z_1^2)(1-z_2^2)(1-z_3)}+\frac{z_3}{(1-z_1^2)(1-z_2)(1-z_3^2)}+\\
+\frac{z_2}{(1-z_1)(1-z_2^2)(1-z_3^2)}+\frac{1+z_1z_2z_3}{(1-z_1^2)(1-z_2^2)(1-z_3^2)}=\\
=\frac{1}{(1-z_1)(1-z_2)(1-z_3)}.
\end{multline*}
The nontrivial dual points are
\begin{align*}
\WG{1}\colon &\ (1,-1,1),\ (-1,1,1),\ (-1,-1,1)\\
\WG{2}\colon &\ (1,1,-1),\ (-1,1,1),\ (-1,1,-1)\\
\WG{3}\colon &\ (1,1,-1),\ (1,-1,1),\ (1,-1,-1)\\
\WG{4}\colon &\ (1,-1,-1),\ (-1,1,-1),\ (-1,-1,1).
\end{align*}
We see that each dual point occurs precisely twice. This is a lattice tiling
in dimension $3$ without the translational property.
\end{example}

\begin{corollary}\label{cor:detequal}
Assume that the point $\z\neq\mathbf{1}$ is a dual point of the sublattices $\LL_i$ and $\LL_j$ and no other sublattices. Then $\det \LL_i=\det \LL_j$. 
\end{corollary}
\begin{proof}
By \eqref{eq:mainres} we get
\[\frac{\Chi_\z(\v_i)}{\det \LL_i}+\frac{\Chi_\z(\v_j)}{\det \LL_j}=0.\]
But the numerators are roots of unity, so the denominators, being both
positive integers, must also agree.
\end{proof}

\subsection{Sufficiency of the residue condition}

We will show now that the conditions given by Proposition~\ref{prop:mainequality} are sufficient for the tiling. More precisely,
we have the following result.

\begin{theorem}\label{thm:sufficient}
Let $\v_1+\LL_1,\dots,\v_n+\LL_n$ be sublattice translates in $\Z^d$ with generating functions $\Theta_1,\dots,\Theta_n$. Suppose that
for any $\z\in\T^d$ we have the relation \eqref{eq:mainres}, then the
sublattice translates tile $\Z^d$.
\end{theorem}
\begin{remark}
If $\z$ is not a dual point of any of the sublattices, then \eqref{eq:mainres} is an empty relation. Therefore it is enough to check \eqref{eq:mainres}
for finitely many cases.
\end{remark}
\begin{proof}
We would like to show that $\Theta_1,\dots,\Theta_n$ satisfy \eqref{eq:fund}.
In the following, for a polynomial $P$ in variables $z_1,\dots,z_d$, we write $\deg_k P$ as the degree in variable $z_k$.

Observe that each generating function $\Theta_j$ vanishes at infinity. More precisely, if we fix $z_1,\dots,\widehat{z}_k,\dots,z_d$
such that $z_m^{t_{j,m}}\neq 1$ for any $m\neq k$ (where $t_{j,m}$ denotes
the $m$-th polar value of $\LL_j$), then we have $\lim_{z_k\to\infty}\Theta_j(z_1,\dots,z_d)=0$.
This is a direct consequence of the fact that $\deg_kR_j<t_{j,k}$. In particular if we define
\[\Theta(\z)=\Theta_1(\z)+\dots+\Theta_n(\z)-\frac{1}{(1-z_1)\dots(1-z_d)},\]
then $\Theta$ also vanishes at infinity. We can write $\Theta$ in the following way.
\[\Theta(\z)=\frac{R(\z)}{Q_1(z_1)\dots Q_d(z_d)},\]
where $R$ is a polynomial in $z_1,\ldots,z_d$ and $Q_m$ is the least common
multiple of $1 - z_m^{t_{1,m}},\dots,1 - z_m^{t_{n,m}}$. 
Notice that each $Q_m$ is square free.
The asymptotics
of $\Theta$ implies that
\begin{equation}\label{eq:asymptotics}
\deg_kR<\deg Q_k \quad \text{for } 1 \le k \le d.
\end{equation}

We first claim that if $u_1,\dots,u_d$ are such that
$Q_1(u_1)=\dots=Q_d(u_d)=0$, then we have $R(u_1,\dots,u_d)=0$. Indeed, if $\u=(u_1,\dots,u_d)$ is not a dual
point of any $\LL_j$, then the residue of any $\Theta_j(\z)d\z$ at $\u$ is
zero (see proof of Lemma~\ref{lem:stupidintegral}). If it is a dual point of some
sublattices, then the second case of \eqref{eq:mainres} applies.
Therefore, $R(u_1,\dots,u_d)$ is proportional to the total residue of the form
$\Theta(\z)d\z$ at $u_1,\dots,u_d$, which is zero.

Now by induction we shall show that for any $k$, and any $u_{k+1},\dots,u_d$ satisfying $Q_{k+1}(u_{k+1})=\dots=Q_{d}(u_d)=0$, we have
\[R(z_1,\dots,z_k,u_{k+1},\dots,u_d)\equiv 0\textrm{\ \ as a polynomial in $z_1,\dots,z_k$}.\]
The induction assumption (for $k=0$) is done. Now suppose we have proved it
for $k-1$. Thus, the polynomial
\[P_k(z_1,\dots,z_{k-1},z_k):=R(z_1,\dots,z_k,u_{k+1},\dots,u_d)\]
vanishes at any $z_k=u_k$ that satisfy $Q_k(u_k)=Q_{k+1}(u_{k+1}) = \dots =
Q_d(u_d) = 0$. Hence, $P_k$ is divisible by $(z_k-w_1)\ldots(z_k-w_l)$,
where $w_1,\ldots,w_l$ are roots of $Q_k$. Since $Q_k$ is square free, we have
$l=\deg Q_k$. But $\deg_k P_k<\deg Q_k$. The only possibility is that
$P_k\equiv 0$. The induction step is done.

The statement for $k=d$ implies that $R$ is identically zero. This is equivalent to \eqref{eq:fund}, and the proof is finished.
\end{proof}
\begin{remark}
The above proof is a generalization of the fact that if a rational function on $\C$ has only simple poles, vanishes at infinity and has residue $0$ at
each pole, then it is equal to zero everywhere. One could express the above proof in the language of multidimensional residues, but we wanted the proofs to be accessible to non-experts.
\end{remark}

\bigskip

\section{The translational property}\label{sec:translational}

\subsection{The translational property in $\Z^d$}
A lattice tiling has the translational property if at least two of its
translated sublattices are different cosets of the same sublattice.

The translational property is known in $\Z^1$ (see
Section~\ref{sec:background}). In fact, for
any expression of $\Z^1$ as a nontrivial \emph{rational linear
  combination} of arithmetic progressions, there are two progressions having same difference.

In \cite{Tiling1}, a counterexample to the translational property in $\Z^3$
was given. 
We reproduced this in Example~\ref{ex:nonstandard}.
This construction directly generalizes to higher dimensions.
We now focus on the translational property in lattice tilings of $\Z^2$. 
But first, we want to point out a counterexample for Question~\ref{question3} in
a more general setting, in which sublattice translates carry rational weights and
are allowed to overlap.

\begin{example}
Consider $\LL_1 = \text{span} ((2,0), (0,1)) $, $\LL_2 = \text{span} ((1,0),
(0,2)), \LL_3 = \text{span} ( (2,0)$, $(0,2) )$, $\LL_4 = \text{span} ( (1,1),
(1,-1) )$. These sublattices of $\Z^2$ are obviously distinct. However, we have a tiling:
\[
\Z^2 = \frac{1}{2}(\LL_1 + (0,0)) + \frac{1}{2}(\LL_2 + (0,0)) + 1(\LL_3 +
(1,1)) + \frac{1}{2}(\LL_4 + (1,0)). 
\]

Here the rational coefficient in front of each translate $(\LL_i + \v_i)$ tells us
how many ``times'' a point in that translate is counted.
\end{example}

This implies that relaxing the disjointness condition will destroy the
combinatorial rigidity of lattice tilings. 
Before investigating the  translational property in $\Z^2$, we look at some
properties of 2-dimensional dual groups.

\subsection{Characterisation of dual groups of sublattices in $\Z^2$}
It is easy to show that a 2-dimensional sublattice $\LL$ generated by $v_1=(a,b)$ and $v_2=(c,d)$ is cyclic if and only if $\gcd(a,b,c,d)=1$. The quantity $e=\gcd(a,b,c,d)$ does not depend on the choice of basis and we will call it the \emph{multiplicity} of $\LL$. 

It follows from the definition of Smith Normal Form  
that we have an isomorphism, namely
$\WGL \cong \Z_e \oplus \Z_\frac{\det\LL}{e}$ and furthermore
$e^2|\det\LL$. In particular, if $\det\LL$ is square free, the sublattice $\LL$ is necessarily cyclic.
In dimensions higher than $2$, cyclicity is more subtle and the group $\WGL$ might be more complicated.

To stress the difference between the 2-dimensional case and the higher dimensional ones, we first prove a simple result.

\begin{lemma}\label{lem:onlytwo}
Let $\LL_1$ and $\LL_2$ be two sublattices of $\Z^2$ with equal multiplicities $e_1=e_2=e$. 
Assume that there is a common vector $w=(a,b)\in \LL_1\cap \LL_2$ such that $\gcd(a,b)=e$.
Assume also that $\det \LL_1=\det \LL_2$, then we have $\LL_1=\LL_2$.
\end{lemma}
\begin{proof}
Rescaling the two sublattices by the factor $1/e$ we can assume that $e=1$. 
Let $\w=(a,b)$ and $\v_1=(c_1,d_1)$, $\v_2=(c_2,d_2)$ be two vectors such that $(\w,\v_i)$ spans $\LL_i$ and $ad_i-bc_i=\det \LL_i$, $i=1,2$. The equality
of determinants implies that
\[a(d_1-d_2)=b(c_1-c_2).\]
As $\gcd(a,b)=1$, we infer that there exists $k\in\Z$ such that $c_1-c_2=ka$, $d_1-d_2=kb$. This means that $\v_2=\v_1+k\w$.
In particular $\v_2\in\LL_1$, so $\LL_2\subseteq\LL_1$. Similarly $\LL_1
\subseteq \LL_2$ and we are done.
\end{proof}

\begin{remark}
The proof does not work if we do not assume that $\gcd(a,b)=e$. For example
consider $\LL_1=2\Z\times \Z$ and $\LL_2=\Z\times 2\Z$. Then $e=1$, $(2,2)\in \LL_1\cap \LL_2$
and $\det \LL_1=\det \LL_2=2$, but $\LL_1$ and $\LL_2$ are different. In the language of dual points and dual groups we can reformulate Lemma~\ref{lem:onlytwo} as follows.
\end{remark}

\begin{corollary}\label{cor:ident-two}
Given two sublattices $\LL_1$ and $\LL_2$ with the same multiplicity $e$ and determinant $\Delta$. Assume  that there exists $g_1\in \WG{1}$
and $g_2\cap\WG{2}$, both of order $\Delta/e$, such that $g_1^e=g_2^e$. Then $\LL_1=\LL_2$.
\end{corollary}
\begin{proof}
Let $g=g_1^e=g_2^e=(z_1,z_2)\in \T^2$. Consider the two integer sublattices
$\frac{1}{e}\LL_1$ and $\frac{1}{e}\LL_2$. They both are cyclic with determinant 
$\Delta/e^2$ and admit $g$ as a common dual point. Now $g$ also has order $\Delta/e^2$, hence it is a generator 
for both $\widehat{G}_{\frac{1}{e}\mathcal{L}_1}$ and $\widehat{G}_{\frac{1}{e}\mathcal{L}_2}$. So we must have $\widehat{G}_{\frac{1}{e}\mathcal{L}_1}=\widehat{G}_{\frac{1}{e}\mathcal{L}_2}$. The result follows easily by applying Corollary \ref{cor:subset} and then rescaling to the original sublattices.
\end{proof}

\subsection{The translational property in $\Z^2$}\label{tilings2}
In this subsection, we prove some sufficient conditions for the translational
property in $\Z^2$ to hold.
Assume that we are given sublattice translates $\v_1+\LL_1, \dots, \v_n+\LL_n$
that together tile $\Z^2$.
Let us reorder the sublattices in the following way.
\begin{itemize}
\item[(a)] If $i<j$, then $\frac{1}{e_i}\det \LL_i\ge \frac{1}{e_j}\det \LL_j$. In other words, the maximal cyclic subgroup of $\WG{i}$ has at least the same
order as the maximal cyclic subgroup of $\WG{j}$.
\item[(b)] If $i<j$, but $\frac{1}{e_i}\det \LL_i=\frac{1}{e_j}\det \LL_j$,
  then $e_i > e_j$.
\end{itemize}

\begin{theorem}\label{thm:main-two}
If the number $e_1$ is of the form $p^r$ for $p$ a prime, then the tiling has the translation property.
\end{theorem}

\begin{proof}
We set $\alpha=\det\LL_1/e_1$.  Let $\z\in\WG{1}$ be an element of order $\alpha$.
We define a sequence $1=n_{1}<n_{2}<\dots<n_{s}$ of indices such that $\z$ belongs to $\WG{n_{1}},\dots,\WG{n_{s}}$ and to no other
lattice. This sequence is nonempty by Corollary~\ref{cor:comeinpairs}. To shorten the notation we will write $\LL_k^\z$, $\v_k^\z$ instead of $\LL_{n_k}$, $\v_{n_k}$.

The maximum order over all elements in $\WGR_{\LL_k^\z}$ is $\det\LL_k^\z/e_k^\z$. By the ordering condition we have $\det\LL_k^\z/e_k^\z\le\alpha$.
But $\WGR_{\LL_k^\z}$ contains the element $\z$ of order $\alpha$. Hence
$\det\LL_k^\z/e_k^\z=\alpha$. Let us now apply now Proposition~\ref{prop:mainequality} to get the following equation
\begin{equation}\label{eq:cor}
\sum_{k=1}^{s}\frac{\Chi_{\z}(\v^\z_k)}{\alpha e^\z_k}=0,
\end{equation}
where we wrote $\det\LL_k^\z=\alpha e^\z_k$. Each term $\Chi_{\z}(\v^\z_k)$
is a root of unity. Define
\[a_k=\Chi_{\z}(\v^\z_1)^{-1}\Chi_{\z}(\v^\z_k).\]
Equation \eqref{eq:cor} now takes the following form:
\begin{equation}\label{eq:finalform}
\sum_{k=1}^{s}\frac{a_{k}}{e^\z_k}=0.
\end{equation}

Denote by $g$ the minimal positive integer such that $a_{k}^g=1$ for all
$k=1,\dots,s$. Next we need two lemmas.

\begin{lemma}\label{lem:twoaredifferent}
If there are  $k \neq l$ such that $e_k^\z=e_l^\z$, then the translational
property holds.
\end{lemma}
\begin{proof}[Proof of Lemma~\ref{lem:twoaredifferent}]
The sublattices $\LL_k^\z$ and $\LL_l^\z$ have the same determinant and multiplicity, and share an
element $\z$ of order equal to the order of each sublattice. By
Corollary~\ref{cor:ident-two} we obtain that $\LL_k^\z=\LL_l^\z$. So the
translational property holds.
\end{proof}

\begin{lemma}\label{lem:str}
Suppose that there exists a prime $q$ an integer $l>0$, and an index $k\in\{1,\dots,s\}$ such that $q^l|e^\z_k$.
Then there exists $k'\in\{1,\dots,s\}$, $k'\neq k$ such that $q^l|e^\z_{k'}$.
\end{lemma}
\begin{proof}[Proof of Lemma~\ref{lem:str}]
Assume the contrary, so that $k$ is the unique index for which $q^l|e^{\z}_k$.
Let $B$ be the least common multiple of $e^{\z}_1,\dots,e^\z_{s}$ and $B_k=B/e^{\z}_{k}$.
By the above assumption, $q\not|B_k$ and for any $n\neq k$ we have $q|B_n$.

Equation \eqref{eq:finalform} can be now be rewritten as
\begin{equation}\label{eq:ultimateform}
B_k + \sum_{n\neq k}B_n \frac{a_n}{a_k} = B_k+\sum_{n\neq k}B_m\e^{\gamma_n}=0,
\end{equation}
where $\e$ is a root of unity of order $g$, and $\gamma_n\in\{0,\dots,g-1\}$.
The above expression is a polynomial in $\e$.
Let us denote this polynomial by $P(\e)$.  By the assumption on $B_1,\ldots,B_s$, we have.
\[P(\e)=B_k+qQ(\e),\]
where $Q$ is a polynomial with integer coefficients. 

Now let $H$ be the minimal integer polynomial for a $g$-th root of
unity. This is a monic, symmetric polynomial.
Since $P(\e)=0$, $H$ divides $P$. Since $H$ is monic, the quotient $R=P/H\in\Z[x]$ has integer coefficients.
We end up with the following relation in the ring $\Z[x]$:
\begin{equation}\label{eq:Bj}
B_k+qQ(x)=R(x)H(x).
\end{equation}
Let us now reduce this equation modulo $q$. We get $B_k\bmod q=(R\bmod q)(H\bmod q)$, where $(H\bmod q)$ has positive degree (because $H$
is monic) and 
$B_k\not\equiv 0\bmod q$.
This cannot hold, for either $\deg(R\bmod q)\geq 0$ and the right-hand side
has positive degree, or $R\equiv 0\bmod q$ so the left-hand side must be zero.
\end{proof}

\begin{remark}\label{rem:str}
We point out that Lemma~\ref{lem:str} works without any assumption on the
translational property. It is a direct consequence of 
Proposition~\ref{prop:mainequality}, that is of the tiling condition. The result is valid also in higher dimensions, if we define $e$ as the quotient
of the determinant over the order.
\end{remark}

\smallskip
\noindent\emph{Conclusion of the proof of Theorem~\ref{thm:main-two}}

We apply now Lemma~\ref{lem:str} to $q^l=p^r=e_1^\z=e_1$. We find another index $k>1$ such that $p^r|e^\z_k$. But $e^\z_k\leq e_1=p^r$ by the
ordering, so we must have $e^\z_k=e_1$. Now Lemma~\ref{lem:twoaredifferent}
ensures the translational property.

\end{proof}

\bigskip
\section{Open questions}\label{sec:open}
We end with some open questions which arise naturally from the results above.  

\begin{problem}\label{problemifandonlyiftranslationtiling}
For dimensions $d \geq 2$, give a necessary and sufficient condition,  in terms of the arithmetic of the sublattice translates, for a lattice tiling to possess the translation property.
\end{problem}

\smallskip
We call a lattice tiling \emph{primitive} if it is not a split tiling. 
In other words, a primitive lattice tiling is a tiling that cannot be formed by splitting another tiling which has a smaller number of sublattice translates.

\begin{problem}\label{primitivetilings}
What are the divisibility relations between the determinants of the
sublattices in a primitive tiling?
\end{problem}

\medskip
And, of course, perhaps the most surprising state of affairs is that Question \ref{question2} from the introduction, concerning two dimensional lattice tilings,  remains open.

\end{document}